\documentclass[12pt]{amsart}
\usepackage{amsmath,amsfonts,amssymb,amsthm}
\usepackage{anysize}
\marginsize{2cm}{2cm}{2cm}{2cm}
\usepackage{graphicx}

\usepackage{color}
\usepackage{ textcomp }
\usepackage{schemata}

\def\z{\mathfrak{z}}
\def\u{\mathfrak{u}}
\def\k{\mathfrak{k}}
\def\g{\mathfrak{g}}
\def\h{\mathfrak{h}}

\def\C{\mathbb{C}}
\def\R{\mathbb{R}}
\def\Q{\mathbb{Q}}
\def\Z{\mathbb{Z}}

\def\H{\mathbb H}

\def\ad{\operatorname{ad}}

\def\I{\operatorname{Id}}
\def\alt{\raise1pt\hbox{$\bigwedge$}}
\def\pint{\langle \cdotp,\cdotp \rangle }

\theoremstyle{plain}
\newtheorem{teo}{\bf Theorem}[section]
\newtheorem{cor}[teo]{\bf Corollary}
\newtheorem{prop}[teo]{\bf Proposition}
\newtheorem{lema}[teo]{\bf Lemma}

\theoremstyle{definition}

\newtheorem{ejemplo}[teo]{\bf Example}

\theoremstyle{remark}
\newtheorem{rem}[teo]{\bf Remark}

\newcommand\aff{\mathfrak{aff}}

\title{Locally conformally K\"ahler solvmanifolds: a survey}
\author{A. Andrada}
\address{FaMAF-CIEM, Universidad Nacional de C\'{o}rdoba, Ciudad Universitaria, 5000 C\'{o}rdoba, Argentina}
\email{andrada@famaf.unc.edu.ar}
\author{M. Origlia}
\address{FaMAF-CIEM, Universidad Nacional de C\'{o}rdoba, Ciudad Universitaria, 5000 C\'{o}rdoba, Argentina, and KU Leuven Kulak, BE-8500 Kortrijk, Belgium}
\email{origlia@famaf.unc.edu.ar}

\thanks{Both authors were partially supported by CONICET, ANPCyT and SECyT-UNC, from Argentina, and the second-named author was also supported by the Research Foundation Flanders (Project G.0F93.17N), from Belgium.}

\keywords{Locally conformally K\"ahler manifold, solvable Lie group, solvmanifold.}
\subjclass[2010]{53B35, 53A30, 22E25}
\date{}

\begin{document}

\begin{abstract}
A Hermitian structure on a manifold is called locally conformally K\"ahler (LCK) if it locally admits a conformal change which is K\"ahler. In this survey we review recent results of invariant LCK structures on solvmanifolds and present original results regarding the canonical bundle of solvmanifolds equipped with a Vaisman structure, that is, a LCK structure whose associated Lee form is parallel.
\end{abstract}

\maketitle

\section{Introduction}

Solvmanifolds and nilmanifolds, i.e., compact quotients of a simply connected solvable (respectively, nilpotent) Lie group by a discrete subgroup, have been used to provide many examples and counterexamples of manifolds with different kind of geometric structures and it is known that they possess a rich structure. For instance, it was shown in \cite{pal} that the total space of a principal torus bundle over a torus is a 2-step nilmanifold. In complex geometry, we can mention that it was proved in \cite{BG} that a non-toral nilmanifold cannot admit K\"ahler structures. This was later generalized to completely solvable solvmanifolds by Hasegawa in \cite{Has}, where he also proves that K\"ahler solvmanifolds are finite quotients of complex tori which have the structure of a complex torus bundle over a torus. Also, the first example of a compact symplectic manifold that does not admit any K\"ahler metric happens to be a nilmanifold, the well known Kodaira-Thurston manifold, which is a primary Kodaira surface (\cite{Kod,Th}). It is well known that the de Rham cohomology of nilmanifolds and completely solvable solvmanifolds can be computed in terms of invariant forms (according to Nomizu \cite{N} and Hattori \cite{Hat}, respectively), and it is conjectured that the Dolbeault cohomology of nilmanifolds can also be computed using invariant forms (this conjecture has been proved for several particular cases, see \cite{CFGU,CF,Rol,FRR}). In \cite{BDV} it was shown that any nilmanifold with an invariant complex structure has holomorphically trivial canonical bundle, while it is known that this does not happen generally for solvmanifolds (see \cite{FOU}).  Other results regarding complex geometric structures on nilmanifolds and solvmanifolds can be found in \cite{AV,CG,K1,MPPS,U,UV}, among many others.

It is then a natural question to study the existence of (invariant) \textit{locally conformally K\"ahler structures} on nilmanifolds and solvmanifolds. We recall that a locally conformally K\"ahler structure (LCK) on a manifold is a Hermitian structure such that each point has a neighborhood where the metric can be rescaled into a K\"ahler metric. The family of LCK manifolds is very important and very rich (for instance, diagonal Hopf manifolds,  Kodaira surfaces, Kato surfaces, some Oeljeklaus-Toma manifolds admit such structures), and it has been widely studied lately. In particular, there are several recent results about LCK structures on solvmanifolds, and it is our attempt to recollect many of them in this survey. 

\

The outline of this article is as follows. In section $2$ we recall the notion and properties of LCK metrics as well as some very well known results on solvmanifolds. In section $3$ we report general facts about invariant LCK structures on solvmanifolds, with emphasis on the case of nilmanifolds. We also recall all the compact complex surfaces that admit a LCK structure which can be seen as solvmanifolds. 
In section $4$ we recall results of LCK structures on solvmanifolds of special kinds: namely, LCK metrics on solvmanifolds with abelian complex structures, and LCK structures on almost abelian Lie groups.
Next, in section $5$ we review the construction of the Oeljeklaus-Toma manifolds and in particular we exhibit them as solvmanifolds, according to the work of Kasuya.  In section $6$ we review some recent results on Vaisman solvmanifolds and we also present original results about the triviality of the canonical bundle of such a manifold (see Theorem \ref{canonical}). Finally, in section $7$ we discuss briefly locally conformally symplectic solvmanifolds and their properties.

\medskip

\noindent \textbf{Acknowledgements.} 
The authors would like to thank Isabel Dotti for useful suggestions during the preparation of this manuscript.

\

\section{Preliminaries}

Let $(M,J,g)$ be a $2n$-dimensional Hermitian manifold, where $J$ is a complex structure and $g$ is a Hermitian metric, and let $\omega$ denote its fundamental $2$-form, that is, 
$\omega(X,Y)=g(JX,Y)$ for any $X,Y$ vector fields on $M$. The manifold $(M,J,g)$ is called {\it locally conformally K\"ahler} (LCK) if $g$ can be rescaled locally, in a 
neighborhood of any point in $M$, so as to be K\"ahler, i.e., there exists an open covering $\{ U_i\}_{i\in I}$ of $M$ and a family $\{ f_i\}_{i\in I}$ of $C^{\infty}$ functions, 
$f_i:U_i \to \R$, such that each local metric 
\begin{equation}\label{gi} 
g_i=\exp(-f_i)\,g|_{U_i} 
\end{equation} 
is K\"ahler. These manifolds are a natural generalization of the class of K\"ahler manifolds, and they have been much studied since the work of I. Vaisman in the '70s. The main references for locally conformally K\"ahler geometry are \cite{DO}, the more recent reports \cite{O, OV1} and the bibliography therein.

An equivalent characterization of a LCK manifold can be given in terms of the fundamental form $\omega$. Indeed, a Hermitian manifold $(M,J,g)$ is LCK if and only if there exists a closed $1$-form $\theta$ globally defined on $M$ such that 
\begin{equation}\label{lck}
d\omega=\theta\wedge\omega.
\end{equation} 
This closed $1$-form $\theta$ is called the \textit{Lee form} (see \cite{L}). Furthermore, the Lee form $\theta$ is uniquely determined by the following formula: 
\begin{equation}\label{tita} 
\theta=-\frac{1}{n-1}(\delta\omega)\circ J, 
\end{equation} 
where $\delta$ is the codifferential operator and $2n$ is the dimension of $M$. It follows from \cite{Ga} that $(1-n)\theta$ can be identified with the trace of the torsion of the Chern connection. A Hermitian manifold $(M,J,g)$ is called \textit{globally conformally K\"ahler} (GCK) if there exists a $C^{\infty}$ function, $f:M\to\R$, such that the metric $\exp(-f)g$ is K\"ahler, or equivalently, the Lee form is exact. Therefore a simply connected LCK manifold is GCK. 

There is yet another equivalent definition of LCK manifolds (see for instance \cite{DO}). Let $\tilde M$ be the universal covering of the complex manifold $(M,J)$. Then $M$ has a 
compatible LCK metric if and only if there is a representation $\chi:\pi_1(M)\to \R_{>0}$ and a K\"ahler metric $\tilde g$ on $\tilde M$ such that 
\[ 
 \gamma^*(\tilde{g})=\chi(\gamma)\tilde{g}, \qquad \text{for any } \gamma\in \pi_1(M),
\]
i.e., the fundamental group of $M$ (seen as the group of deck transformations of $\tilde M$) acts by homotheties. Note that the LCK metric metric obtained is GCK if and only if 
$\chi\equiv 1$. 

It is well known that LCK manifolds belong to the class $\mathcal{W}_4$ of the Gray-Hervella classification of almost Hermitian manifolds \cite{GH}. Also, an LCK manifold 
$(M,J,g)$ is K\"ahler if and only if $\theta=0$. Indeed, $\theta\wedge\omega=0$ and $\omega$ non degenerate imply $\theta=0$. It is known that if $(M,J,g)$ is a Hermitian manifold 
with $\dim M\ge 6$ such that \eqref{lck} holds for some $1$-form $\theta$, then $\theta$ is automatically closed, and therefore $M$ is LCK.

In contrast to the K\"ahler class, the LCK class is not stable under small deformations. Indeed, it follows from \cite{Bel} that some Inoue surfaces do not admit LCK structures but they are complex deformations of other Inoue surfaces which do admit LCK metrics \cite{Tr}. On the other hand, just as in the K\"ahler case, the LCK class is stable under blowing-up points (see \cite{Tr,Vu1}).

The Hopf manifolds are examples of LCK manifolds, and they are obtained as a quotient of $\C^n-\{0\}$ with the Boothby metric by a discrete subgroup of automorphisms. These manifolds are diffeomorphic to $S^1\times S^{2n-1}$ and for $n\geq 2$ they have first Betti number equal to 1, so that they do not admit any K\"ahler metric. The LCK structures on these Hopf manifolds have a special property, as shown by Vaisman in \cite{V2}. Indeed, the Lee form is parallel with respect to the Levi-Civita connection of the Hermitian metric. The LCK manifolds sharing this property form a distinguished class, which has been much studied since Vaisman's seminal work \cite{Bel,GMO,KS,OV3,OV4,V2,V3}. Indeed, LCK manifolds with parallel Lee form are nowadays called \textit{Vaisman manifolds}.

Vaisman manifolds satisfy stronger topological properties than general LCK manifolds. For instance, a compact Vaisman non-K\"ahler manifold $(M,J,g)$ has $b_1(M)$ odd (\cite{KS,V3}), which implies that such a manifold cannot admit K\"ahler metrics. Moreover, it was proved in \cite[Structure Theorem]{OV3} and \cite[Corollary 3.5]{OV4} that any compact Vaisman manifold admits a Riemannian submersion to a circle such that all fibers are isometric and admit a natural Sasakian structure. It was shown in \cite {Ve} 
that any compact complex submanifold of a Vaisman manifold is Vaisman, as well. In \cite{Bel} the classification of compact complex surfaces admitting a Vaisman structure is given. It is known that a homogeneous LCK manifold is Vaisman when the manifold is compact (\cite{GMO,HK2}) and, more generally, when the manifold is a quotient of a reductive Lie group such that the normalizer of the isotropy group is compact (\cite{ACHK}). In \cite{AHK} it was proved that a simply connected homogeneous Vaisman manifold $G/H$ with $G$ unimodular is isomorphic to a product $\R\times M_1$, where $M_1$ is a  simply  connected  homogeneous Sasakian manifold of a unimodular Lie group, which in turn is a certain $\R$-bundle or $S^1$-bundle over a simply connected homogeneous K\"ahler manifold of a reductive Lie group.

\medskip

Associated to any LCK metric on a Hermitian manifold there is a cohomology $H^*_\theta(M)$ which can be defined as follows. Since the corresponding Lee form $\theta$ is closed, we can deform the de Rham differential $d$ to obtain the adapted differential operator 
\[d_\theta \alpha= d\alpha -\theta\wedge\alpha.\]
This operator satisfies $d_\theta^2=0$, thus it defines the \textit{Morse-Novikov cohomology} $H_\theta^*(M)$ of $M$ relative to the closed $1$-form $\theta$, also called the Lichnerowicz 
cohomology or the adapted cohomology. It is known that if $M$ is a compact oriented $n$-dimensional manifold, then $H_\theta^0(M)= H_\theta^n(M)=0$ for any non exact closed $1$-form $\theta$ (see for 
instance \cite{GL,Ha}). For any LCK structure $(\omega,\theta)$ on $M$, the $2$-form $\omega$ defines a cohomology class $[\omega]_\theta\in H_\theta^2(M)$, since 
$d_\theta\omega=d\omega-\theta\wedge \omega=0$. It was proved in \cite{LLMP} that the Morse-Novikov cohomology associated to the Lee form of a Vaisman structure on a compact manifold vanishes. 

\

\subsection{Some facts on solvmanifolds}

Recall that a discrete subgroup $\Gamma$ of a simply connected Lie group $G$ is called a \textit{lattice} if the quotient $\Gamma\backslash G$ is compact. According to \cite{Mi}, if such a lattice exists then the Lie group must be unimodular. The quotient $\Gamma\backslash G$ is known as a solvmanifold if $G$ is solvable and as a nilmanifold if $G$ is nilpotent, and it is known that $\pi_1(\Gamma\backslash G)\cong \Gamma$. Moreover, the diffeomorphism class of solvmanifolds is determined by the isomorphism class of the corresponding lattices, as the following results show:

\begin{teo}\cite[Theorem 3.6]{R}
Let $G_1$ and $G_2$ be simply connected solvable Lie groups and $\Gamma_i$, $i=1,2$, a lattice in $G_i$. If $f:\Gamma_1\to\Gamma_2$ is an isomorphism, then there exists a 
diffeomorphism $F:G_1\to G_2$ such that
\begin{enumerate}
 \item $F|_{\Gamma_1}=f$,
 \item $F(\gamma g)=f(\gamma)F(g)$, for any $\gamma\in \Gamma_1$ and $g\in G_1$.
\end{enumerate}
\end{teo}

\begin{cor}\cite{Mo} \label{mostow}
 Two solvmanifolds with isomorphic fundamental groups are diffeomorphic.
\end{cor}

\smallskip

It is not easy to determine whether a given unimodular solvable Lie group admits a lattice. However, there is such a criterion for nilpotent Lie groups. Indeed, it was proved by Malcev in \cite{Ma} that a nilpotent Lie group admits a lattice if and only if its Lie algebra has a rational form, i.e. there exists a basis of the Lie algebra such that the corresponding structure constants are all rational. More recently, in \cite{B}, the existence of lattices in simply connected solvable Lie groups up to dimension $6$ was studied. A general result proved by Witte in \cite[Proposition 8.7]{Wi} states that only countably many nonisomorphic simply connected Lie groups admit lattices. 

\medskip

In the case when $G$ is completely solvable, i.e., it is a solvable Lie group such that the endomorphisms $\ad_X$ of its Lie algebra $\g$ have only real eigenvalues for all $X\in \g$, the de Rham and the Morse-Novikov cohomology of $\Gamma\backslash G$ can be computed in terms of the cohomology of $\g$. Indeed, Hattori proved in \cite{Hat} that if $V$ is a finite dimensional triangular\footnote{A $\g$-module $V$ is called triangular if the endomorphisms of $V$ defined by $v\mapsto Xv$ have only real eigenvalues for any $X\in\g$.} $\g$-module, then $\overline{V}:=C^\infty(\Gamma\backslash G)\otimes V$ is a $\mathfrak{X}(\Gamma\backslash G)$-module and there is an isomorphism 
\begin{equation}\label{kill}
 H^*(\g,V) \cong H^*(\mathfrak{X}(\Gamma\backslash G), \overline{V}).
\end{equation}
Therefore:
\begin{itemize}
 \item If $V=\R$ is the trivial $\g$-module, then the right-hand side in \eqref{kill} gives the usual de Rham cohomology of $\Gamma\backslash G$, so that 
\begin{equation}\label{deRham}
H^*(\g) \cong H^*_{dR}(\Gamma\backslash G).
\end{equation}
 \item If $V=V_\theta$, where $\theta$ is a closed left-invariant 1-form and $V_{\theta}$ is a $1$-dimensional $\g$-module with action given by 
\[  Xv=-\theta(X)v, \quad X\in\g,\, v\in V_{\theta}, \]
then we can identify $\overline{V}$ with $C^\infty(\Gamma\backslash G)$ and the action of 
$\mathfrak{X}(\Gamma\backslash G)$ on $C^\infty(\Gamma\backslash G)$ is given by 
\[ X\cdot f= Xf-\theta(X)f, \qquad X\in\g, \, f\in C^\infty(\Gamma\backslash G).\]
Here we are using that there is a natural inclusion $\g \hookrightarrow 
\mathfrak{X}(\Gamma\backslash G)$ and a bijection $C^\infty(\Gamma\backslash G)\otimes \g \to 
\mathfrak{X}(\Gamma\backslash G)$ given by $f\otimes X \mapsto fX$. As a consequence, in this case 
\eqref{kill} becomes (cf. \cite[Corollary 4.1]{Mil})
\begin{equation}\label{kill_adapted}
 H^*_\theta(\g)\cong H^*_\theta(\Gamma\backslash G).
\end{equation}
\end{itemize}
In particular, $H^*_{dR}(\Gamma\backslash G)$ and $H^*_\theta(\Gamma\backslash G)$ do not depend on 
the lattice $\Gamma$.

\medskip
In the general case, i.e. when $G$ is not necessarily completely solvable, there is always an injection $i^*:H^*_\theta(\g)\to H^*_\theta(\Gamma\backslash G)$, where $\theta$ is 
any closed left-invariant 1-form, as shown in \cite{K}.

\

\section{LCK solvmanifolds}

Let $G$ be a Lie group with a left invariant Hermitian structure $(J,g)$. If $(J,g)$ satisfies 
the LCK condition \eqref{lck}, then $(J,g)$ is called a {\em left invariant LCK structure} on the 
Lie group $G$. Clearly, the fundamental $2$-form is left invariant and, using \eqref{tita}, it is 
easy to see that the corresponding Lee form $\theta$ on $G$ is also left invariant. 
 
This fact allows us to define LCK structures on Lie algebras. We recall that a {\em complex
structure J} on a Lie algebra $\g$ is an endomorphism $J: \g \to \g$ satisfying $J^2=-\I$ and 
\[ N_J=0, \quad  \text{where} \quad N_J(x,y)=[Jx,Jy]-[x,y]-J([Jx,y]+[x,Jy]),\]        
for any $x,y \in \g$. 

Let $\g$ be a Lie algebra, $J$ a complex structure and $\pint$ a Hermitian inner product on $\g$, 
with $\omega\in\alt^2\g^*$ the fundamental $2$-form. We say that $(\g,J,\pint)$ is {\em locally 
conformally K\"ahler} (LCK) if there exists $\theta \in \g^*$, with $d\theta=0$, such that
\begin{equation} \label{g-lck-0}
d\omega=\theta\wedge\omega.
\end{equation}
We will assume from now on that $\theta\neq 0$, so that we are excluding the K\"ahler case.

Recall that if $\alpha\in\g^*$ and $\eta\in\alt^2\g^*$, then their exterior derivatives $d\alpha\in\alt^2\g^*$
and $d\eta\in\alt^3\g^*$ are given by
\[ d\alpha(x,y)= -\alpha([x,y]), \qquad
d\eta(x,y,z)=-\eta([x,y],z)-\eta([y,z],x)-\eta([z,x],y),\]
for any $x,y,z\in\g$.

We have the following orthogonal decomposition for a Lie algebra $\g$ with an LCK structure $(J,\pint)$, 
\[\g=\mathbb{R}A \oplus \ker\theta\] 
where $\theta$ is the Lee form and $\theta(A)=1$. Since $d\theta=0$, we have that $\g'=[\g,\g]\subset\ker\theta$. Note also that since $\theta\neq0$, $\g$ cannot be a semisimple Lie algebra. It is clear that $JA\in\ker\theta$, but more can be said when $\g$ is unimodular. Indeed, we proved in \cite{AO} that $JA$ is in the commutator ideal $[\g,\g]$ of $\g$, using \eqref{tita}. 

\

We note that any left invariant LCK structure $(J,g)$ on a Lie group $G$ induces naturally a LCK structure on a quotient $\Gamma\backslash G$ of $G$ by a discrete subgroup $\Gamma$. The induced LCK structure on the quotient will be called \textit{invariant}.

\

\subsection{LCK nilmanifolds}

The first example of an LCK nilmanifold was given in \cite{CFL}. We recall their construction:

\begin{ejemplo}\label{heisenberg}
Let $\g=\R\times\h_{2n+1}$, where $\h_{2n+1}$ is the $(2n+1)$-dimensional
Heisenberg Lie algebra. There is a basis $\{X_1,\dots,X_n,Y_1,\dots,Y_n,Z_1,Z_2\}$ of $\g$ with Lie
brackets given by $[X_i,Y_i]=Z_1$ for $i=1,\dots,n$ and $Z_2$ in the center. We define an inner
product $\pint$ on $\g$ such that the basis above is orthonormal. Let $J_0$ be the almost complex structure on $\g$ given by:
\[J_0X_i=Y_i, \quad  J_0Z_1=-Z_2  \; \; \; \text{for $i=1,\dots,n$}.\]
It is easily seen that $J_0$ is a complex structure on $\g$ compatible with $\pint$. If $\{x^i,y^i,z^1,z^2\}$ denote the $1$-forms dual to $\{X_i,Y_i,Z_1,Z_2\}$ respectively, then the 
fundamental $2$-form is: 
\[\omega=\sum_{i=1}^n(x^i\wedge y^i) - z^1\wedge z^2.\]
Thus, \[ d\omega=z^2\wedge\omega,\]
and therefore $(\g,J_0,\pint)$ is LCK. 

It is known that $\g$ is the Lie algebra of the Lie group $\R\times H_{2n+1}$,
where $H_{2n+1}$ is the $(2n+1)$-dimensional Heisenberg group. The Lie group $H_{2n+1}$ admits a lattice $\Gamma$ and therefore the nilmanifold $N= S^1 \times
\Gamma\backslash H_{2n+1}$ admits an invariant LCK structure which is Vaisman. The nilmanifold $N$ is a primary Kodaira surface and it cannot admit any K\"ahler metric, due to \cite{BG}.
\end{ejemplo}

In \cite{U}, it was proved that if $(J,g)$ is a Hermitian structure on a $6$-dimensional nilmanifold $\Gamma\backslash G$ such that $J$ is invariant, then this structure is LCK if and only if $G$ is isomorphic to $H_5\times \R$, and the complex structure $J$ is equivalent to $J_0$. Moreover, the LCK structure is actually Vaisman. Based on this result, Ugarte states in the same article the following conjecture:

\

\textbf{Conjecture:} A $(2n+2)$-dimensional nilmanifold admitting an LCK structure (not necessarily invariant) is diffeomorphic to a product $\Gamma\backslash H_{2n+1}\times S^1$, 
where $\Gamma$ is a lattice in $H_{2n+1}$.

\

We mention next some recent advances towards the solution of this conjecture.

In \cite{S}, H. Sawai proves the conjecture in the case when the complex structure on the Hermitian nilmanifold is invariant, that is, he extends the result of Ugarte to any even dimension.

\begin{teo}[\cite{S}]
Let $(M,J)$ be a non-toral compact nilmanifold with a left invariant complex structure. If $(M,J)$ admits a locally conformally K\"ahler metric, then $(M,J)$ is biholomorphic to a quotient of $(H_{2n+1}\times \R ,J_0)$.
\end{teo}

In his proof, Sawai uses the fact that the Morse-Novikov cohomology $H^k_{\theta}(\g)$ of a nilpotent Lie algebra $\g$ is trivial for any closed $1$-form $\theta$ and $k\geq 2$.

\

Another weaker version of the conjecture was proved by G. Bazzoni in \cite{Baz}. He considers nilmanifolds equipped with Vaisman structures and he proves:

\begin{teo}\label{nilpotent}
A $(2n+2)$-dimensional nilmanifold admitting a Vaisman structure (not necessarily invariant) is diffeomorphic to a product $\Gamma\backslash H_{2n+1}\times S^1$, where $\Gamma$ is a lattice in $H_{2n+1}$.
\end{teo}

The idea of the proof goes as follows. If $\Gamma\backslash G$ is a compact Vaisman nilmanifold then, due to results of Ornea-Verbitsky, there exists a Sasakian manifold $S$ equipped with a Sasakian automorphism $\varphi: S\to S$ such that $\Gamma\backslash G$ is diffeomorphic to the mapping torus $S_{\varphi}$. Then it is proved that this Sasakian manifold has the same minimal model as a compact nilmanifold $\Lambda\backslash H$, where $\Lambda=\pi_1(S)$. Using some arguments concerning minimal models of nilpotent spaces appearing in \cite{CDMY}, the author proves that the nilpotent Lie group $H$ is isomorphic to $H_{2n+1}$. Standard results on Lie groups and their lattices imply the result.

\

\subsection{4-dimensional LCK solvmanifolds}

In \cite{ACFM}, it is given the first known example of a compact solvmanifold equipped with an LCK structure, which is not diffeomorphic to a nilmanifold. Their construction is given as follows:

\begin{ejemplo}\label{ejemplo-4d}
Let $\g$ be the $4$-dimensional Lie algebra with basis $\{A,X,Y,Z\}$ with Lie brackets
 \[ [A,X]=X,\quad [A,Z]=-Z, \quad [X,Z]=-Y. \]
Define an inner product on $\g$ in such a way that the basis above is orthonormal, and define an almost complex structure $J$ on $\g$ by
\[ JA=Z, \quad JX=Y.\]
It is easy to see that $J$ is integrable, thus it is a complex structure on $\g$. If $\{ \alpha,x,y,z\}$ denotes the dual basis of $\g^*$, then the fundamental $2$-form $\omega$ 
is given by $\omega=\alpha\wedge z+x\wedge y$, and hence $d\omega=-\alpha\wedge \omega$. Consequently, $\g$ admits an LCK structure with Lee form $\theta=-\alpha$. Note that it is 
not Vaisman since $(\nabla_X \theta)(X)=1$. 

It is well known that the corresponding simply connected Lie group $G$ has non-nilpotent lattices, and therefore if $\Gamma\backslash G$ is any associated solvmanifold, then it carries a non-Vaisman LCK structure. It was proved by Kamishima in \cite{Kam} that this invariant LCK structure on the compact complex surface $\Gamma\backslash G$ coincides with the LCK structure constructed by Tricerri on the Inoue surface of type $S^+$ in \cite{Tr}.
\end{ejemplo}

\

As mentioned before, the solvmanifold from Example \ref{ejemplo-4d} can be identified with an Inoue surface of type $S^+$. More generally, it was proved in \cite[Theorem 1]{Ha1} that a compact complex surface $X$ diffeomorphic to a solvmanifold $\Gamma\backslash G$ is either (1) a complex torus, (2) a primary Kodaira surface, (3) a secondary Kodaira surface, (4) an Inoue surface of type $S^0$, (5) an Inoue surface of type $S^+$, or (6) a hyperelliptic surface; moreover, the complex structure on $X$ can be seen to be induced by a left invariant one on $G$. In each case we recall the structure equations for the Lie algebra of $G$, which has a basis $\{A,X,Y,Z\}$:
\begin{enumerate}
 \item complex torus: all brackets vanish,
 \item primary Kodaira surface: $[X,Y]=Z$,
 \item secondary Kodaira surface: $[X,Y]=Z,\, [A,X]=Y,\, [A,Y]=-X$,
 \item Inoue surface of type $S^0$: $[A,X]=-\frac12 X+bY,\, [A,Y]=-bX-\frac12 Y,\, [A,Z]=Z$, $b\in\R$,
 \item Inoue surface of type $S^+$: $[X,Z]=-Y,\, [A,X]=X,\, [A,Z]=-Z$,
  \item hyperelliptic surface: $[A,X]=-Y,\, [A,Y]=X$. 
\end{enumerate}
According to \cite{HK}, all of the $4$-dimensional compact complex surfaces above admit a LCK structure, except for the complex torus and the hyperelliptic surface. More precisely, consider for the cases (2)-(5) the complex structure $J$ given by $JX=Y,\, JA=Z$ and the metric such that the basis $\{A,X,Y,Z\}$ is orthonormal. If $\{\alpha,x,y,z\}$ denotes the dual basis of this basis then the fundamental $2$-form $\omega$ is given by $\omega=\alpha\wedge z+x\wedge y$, and the Lee form is $\theta=\alpha$ for cases (2)-(4) and $\theta=-\alpha$ for the case (5).

\medskip

More recently, in \cite{AngO} a complete classification of locally conformally K\"ahler structures on $4$-dimensional solvable Lie algebras (not only unimodular) up to linear equivalence was given.

\begin{rem}
For the sake of completeness, we recall that LCK structures on $4$-dimensional reductive Lie algebras were studied in \cite{ACHK}. The authors prove that if a reductive Lie algebra admits a LCK structure then the Lie algebra is isomorphic to either $\mathfrak{gl}(2,\R)$ or $\mathfrak{u}(2)$. In particular, such examples occur only in dimension four. A compact quotient of $\operatorname{GL}(2,\R)$ by a discrete subgroup corresponds to a properly elliptic surface, while $\operatorname{U}(2)$ with an invariant complex structure corresponds to a Hopf surface.
\end{rem}

\

\section{Special cases}

In this section we review results concerning special cases of LCK structures on solvable Lie algebras. First we consider LCK structures where the complex structure is abelian, and later we analyze the existence of LCK structures on almost abelian Lie algebras.

\medskip

\subsection{LCK structures with abelian complex structure}

An almost complex structure $J$ on $\g$ is called \textit {abelian} if 
\[ [JX,JY]=[X,Y], \, \, \text{for all} \, \, 
X,Y\in \g.\]
Note that an abelian almost complex structure is automatically integrable, and therefore it will be called an abelian complex structure.

These complex structures have had interesting applications in differential geometry. For instance, a pair of anticommuting abelian complex structures on $\g$ gives rise to an invariant weak HKT structure on any Lie group $G$ associated to $\g$ (see \cite{DF} and \cite{GP}). In \cite{CF} it has been shown that the Dolbeault cohomology of a nilmanifold with an abelian complex structure can be computed algebraically. Also, deformations of abelian complex structures on nilmanifolds have been studied in \cite{CFP, MPPS}. More recently, it was shown in \cite{UV} that the holonomy group of the Bismut connection of an invariant Hermitian structure on a $6$-dimensional nilmanifold reduces to a proper subgroup of $SU(3)$ if and only if the complex structure is abelian. This result was later extended in \cite{AV} to solvmanifolds of any dimension equipped with abelian balanced Hermitian structures.

\medskip

Some properties of this kind of complex structures are stated in the following result (see \cite{ABD,BD} for their proofs).

\begin{lema}\label{prop}
Let $\g$ be a Lie algebra with $\z(\g)$ its center and $\g':=[\g,\g]$ its commutator ideal. If $J$ is an abelian complex structure on $\g$, then
\begin{enumerate}
\item $J\z(\g)=\z(\g)$.
\item $\g' \cap J\g'\subset \z(\g'+J\g')$.
\item The codimension of $\g'$ is at least $2$, unless $\g$ is isomorphic to $\aff(\R)$ (the only $2$-di\-men\-sional non-abelian Lie algebra).
\item $\g'$ is abelian, therefore $\g$ is $2$-step solvable.
\end{enumerate}
\end{lema}

\

We recall next the main result from \cite{AO}, concerning Lie algebras equipped with a LCK structure with abelian complex structure. Before stating the main result, we consider 
the following variation of Example \ref{heisenberg}.
Recall that $\R\times\mathfrak{h}_{2n+1}$ has a basis $\{X_1,\dots,X_n,Y_1,\dots,Y_n,Z_1,Z_2\}$
such that $[X_i,Y_i]=Z_1$ for $i=1,\dots,n$, and that this Lie algebra admits an abelian complex
structure $J_0$ given by $J_0X_i=Y_i,\, J_0Z_1=-Z_2$. For any $\lambda>0$, consider the metric
$\pint_\lambda$ such that the basis above is orthogonal, with $|X_i|=|Y_i|=1$ but
$|Z_1|^2=|Z_2|^2=\frac{1}{\lambda}$. It is easy
to see (just as in Example \ref{heisenberg}) that $(J_0,\pint_\lambda)$ is an LCK structure, in fact,
it is Vaisman. Furthermore, the metrics $\pint_\lambda$ are pairwise non-isometric, since the scalar
curvature of $\pint_\lambda$ is $-\frac{n\lambda^2}{2}$. 

\

\begin{teo}[\cite{AO}]
Let $(J,\pint)$ be an LCK structure on $\g$ with abelian complex structure $J$. If $\g$ is unimodular then $\g \simeq \R\times\h_{2n+1}$, where $\mathfrak{h}_{2n+1}$ is the $(2n+1)$-dimensional Heisenberg Lie algebra, and $(J,\pint)$ is equivalent to $(J_0,\pint_\lambda)$ for some $\lambda>0$.
\end{teo}

\

In consequence, up to isometry and rescaling, the only examples of LCK structures on Lie algebras with abelian complex structures are those from Example \ref{heisenberg}, and they are Vaisman.

\medskip

\subsection{LCK structures on almost abelian solvmanifolds}

We recall that a Lie group $G$ is said to be \textit{almost abelian} if its Lie algebra $\g$ has a codimension one abelian ideal. Such a Lie algebra will be called almost abelian, 
and it can be written as $\g= \R f_1 \ltimes_{\ad_{f_1}} \mathfrak{u}$, where $\mathfrak u$ is an abelian ideal of $\g$, and $\R$ is generated by $f_1$. Accordingly, the Lie group 
$G$ is a semidirect product $G=\R\ltimes_\varphi \R^d$ for some $d\in\mathbb N$, where the action is given by $\varphi(t)=e^{t\ad_{f_1}}$.

\

In \cite{AO1} we considered the existence of invariant LCK structures on solvmanifolds associated to almost abelian Lie groups. Firstly, we showed that there are plenty of almost 
abelian Lie algebras which admit LCK structures. Indeed, they are completely characterized in the next result:

\begin{teo}[\cite{AO1}]\label{lcK}
Let $\g$ be a $(2n+2)$-dimensional almost abelian Lie algebra and $(J,\pint)$ a Hermitian
structure on $\g$, and let $\g'$ denote the commutator ideal $[\g,\g]$ of $\g$.
\begin{enumerate}
 \item If $\dim \g'=1$, then $(J,\pint)$ is LCK if and only if $\g$ is isomorphic to 
$\mathfrak{h}_3 \times\R$ or $\mathfrak{aff}(\R)\times \R^2$ as above.
 \item If $\dim\g'\geq 2$, then $(J,\pint)$ is LCK if and only if $\g$ can 
be decomposed as $\g={\mathfrak a}^\perp\ltimes \mathfrak a$, a $J$-invariant orthogonal sum with a
codimension $2$ abelian ideal $\mathfrak a$, and there exists an orthonormal basis $\{f_1,f_2\}$ of 
${\mathfrak a}^\perp$ such that 
\[ [f_1,f_2]=\mu f_2, \quad f_2=Jf_1, \quad \ad_{f_2}|_{\mathfrak a}=0 \text{ and } 
\ad_{f_1}|_{\mathfrak a}= \lambda I + B, \]
for some $\mu,\lambda\in\R$, $\lambda\neq 0$, and $B\in \mathfrak{u}(n)$. The corresponding Lee 
form is given by $\theta=-2\lambda f^1$. Furthermore, the Lie algebra $\g$ is unimodular if and only 
if $\lambda=-\frac{\mu}{2n}$.
\end{enumerate}
\end{teo}

\medskip

We note that the left invariant LCK structures obtained on the Lie groups corresponding to the Lie algebras in Theorem \ref{lcK} (2) or $\mathfrak{aff}(\R)\times \R^2$ from 
Theorem \ref{lcK} (1) are never Vaisman. This follows from Lemma \ref{ad_A-antisim} below. On the other hand, any LCK structure on $\mathfrak{h}_3\times \R$ from Theorem \ref{lcK} (1) is Vaisman (see 
\cite{AO,S}).

\

However, almost abelian solvmanifolds equipped with invariant LCK structures are very scarce, since we have shown in \cite{AO1} that they only occur in dimension 4.

\begin{teo}\label{dim-6}
If $G$ is as above with $\mu\neq0$ and $\dim G \geq 6$, i.e. $n\geq 2$, then $G$ admits no lattice.
\end{teo}

This theorem is a consequence of the following result about the roots of a certain class of polynomials with integer coefficients. 

\begin{lema}[\cite{AO1}]\label{integer-polynomial}
Let $p$ be a polynomial of the form \[p(x)=x^{2n+1}-m_{2n}x^{2n}+m_{2n-1}x^{2n-1}+\cdots+m_1x-1\]
with $m_j\in\Z$ and $n\geq 2$, and let $x_0,\dots,x_{2n}\in \C$ denote the roots of $p$. If 
$x_0\in\R$ is a simple root and $|x_1|=\dots=|x_{2n}|$, then $x_0=1$ and $|x_j|=1$, 
$j=1,\ldots, 2n$.
\end{lema}

\medskip

Indeed, the proof of Theorem \ref{dim-6} goes as follows. If $\g=\R f_1\ltimes \R^{2n+1}$, then the Lie group $G$ is a semidirect product $G=\R\ltimes_{\varphi} \R^{2n+1}$, where 
$\varphi(t)=\exp(t\ad_{f_1}|_{\R^{2n+1}})$, $t\in\R$. Let us assume that $G$ admits lattices. According to \cite{B}, there exists a $t_0\neq 0$ such that $\varphi(t_0)$ can be conjugated to an integer matrix. Therefore, taking into account that $\dim G\geq 6$, it follows from Theorem \ref{lcK} that the characteristic polynomial of $\varphi(t_0)$ satisfies the hypothesis of Lemma \ref{integer-polynomial}, hence its only simple root $e^{t_0\mu}$ is equal to 1, equivalently, $t_0\mu=0$. Since $\mu\neq 0$ (otherwise the Hermitian structure would be K\"ahler) we obtain a contradiction. 

\begin{cor}
In  dimensions at least 6 there are no almost abelian solvmanifolds equipped with invariant LCK structures. 
\end{cor}

\

In dimension 4, a unimodular almost abelian Lie algebra admitting an LCK structure is isomorphic either to $\h_3\times\R$ or to $\g_b$, for $b\geq 0$, where $\g_b=\R f_1\ltimes 
\R^3$ and the adjoint action is given by
\[ \ad_{f_1}|_{\R^3}=\begin{pmatrix}
                   1 \\
  & -\frac{1}{2}& -b\\
  & b &  -\frac{1}{2}
                 \end{pmatrix}. \]
It was shown in \cite{AO1} that the associated simply connected almost abelian Lie groups $G_b$ admit lattices for countably many values of the parameter $b$. As mentioned before, it is known that the solvmanifolds associated to $\g_b$ for these values of $b$ are Inoue surfaces of type $S^0$ \cite{Kam, Tr} (see \cite{S1} for an explicit construction of a lattice on some of the Lie groups $G_b$).

\

\section{Oeljeklaus-Toma manifolds}

In this section we review the construction of the Oeljeklaus-Toma manifolds (OT manifolds for short), which appeared in \cite{OT}. These non-K\"ahler compact complex manifolds arise from certain number fields, and they can be considered as generalizations of the Inoue sufaces of type $S^0$.

We recall briefly their construction. Let $K$ denote a number field of degree $n=[K:\Q]$ with $s>0$ real embeddings $\sigma_1,\ldots,\sigma_s$ and $2t>0$ complex embeddings 
$\sigma_{s+1},\ldots, \sigma_{s+2t}$ into $\C$, with $\sigma_{s+i+t}=\overline{\sigma}_{s+i}$ for $i=1,\ldots,t$; therefore $n=s+2t$. For any choice of natural numbers $s,t$, such a field always exists, according to \cite[Remark 1.1]{OT}. 

Let $\mathcal{O}_K$ denote the ring of algebraic integers of $K$, $\mathcal{O}^\times_K$ its multiplicative group of units and set
\[ \mathcal{O}^{\times,+}_K:=\{ a\in \mathcal{O}^\times_K \mid \sigma_i(a)>0,\, i=1,\ldots,s\}.\]
If $\H$ denotes the upper complex half-plane, Oeljeklaus and Toma proved, using the Dirichlet's units theorem, that $\mathcal{O}^{\times,+}_K\ltimes  \mathcal{O}_K$ acts freely on $\H^s\times \C^t$ by
\begin{align*}
(a,b)\cdot & (x_1+iy_1,\ldots,x_s+iy_s,z_1,\ldots,z_t) = \\
           & = (\sigma_1(a)x_1+\sigma_1(b)+i\sigma_1(a)y_1,\ldots, \sigma_s(a)x_s+\sigma_s(b)+i\sigma_s(a)y_s, \\ 
           & \qquad \quad \sigma_{s+1}(a)z_1+\sigma_{s+1}(b),\ldots,\sigma_{s+t}(a)z_t+\sigma_{s+t}(b)).
\end{align*}
Moreover, there exists an \textit{admissible} subgroup $U\subset \mathcal{O}^{\times,+}_K$ such that the action of $U\ltimes \mathcal{O}_K$ is properly discontinuous. For $t=1$, every subgroup of finite index in $\mathcal{O}^{\times,+}_K$ is admissible.

The manifold $(\H^s\times \C^t)/U\ltimes \mathcal{O}_K$ is called an OT manifold of type $(s,t)$ and it is usually denoted $X(K,U)$. It is a compact complex manifold of (complex) dimension $s+t$. These manifolds have many interesting properties; for instance, it was proved in \cite{OT} that  the following holomorphic bundles over an OT manifold $X=X(K,U)$ are flat and admit no global holomorphic section: the bundle of
holomorphic $1$-forms $\Omega^1_X$, the holomorphic tangent bundle $\mathcal{T}_X$ and any positive power $K_X^{\otimes k}$ of the canonical bundle. Moreover, the Kodaira dimension of any OT manifold is $-\infty$ and they can never be K\"ahler.

\

Concerning Hermitian metrics compatible with the complex structure of OT manifolds, the following result is well known when $t=1$:

\begin{prop}[\cite{OT}]
For $s>0$, any OT manifold of type $(s,1)$ admits LCK metrics. 
\end{prop}

Indeed, consider the following potential
\[ 
 F: \H^s\times \C\to \R,\qquad F(z):=\frac{1}{\prod_{j=1}^s(i(z_j-\overline{z}_j))}+|z_{s+1}|^2.
\]
Then $\omega:=i\partial\overline{\partial}F$ gives the desired K\"ahler form on $\H^s\times \C$.

\medskip

Note that for $s = t = 1$ and $U= \mathcal{O}^{\times,+}_K$, the OT manifold $X(K,U)$ is an Inoue surface of type $S^+$ with the metric given in \cite{Tr}.

\

In \cite{K} Kasuya proved another interesting feature of OT manifolds. Indeed, he proved:

\begin{prop}[\cite{K}]
OT manifolds are solvmanifolds.
\end{prop}

In order to prove this, Kasuya considers the Lie group $G=\R^s\ltimes_\phi(\R^s\times \C^t)$, where the action of $\phi$ can be described as follows. For $a\in U$, set $t_i=\log|\sigma_i(a)|$, $i=1,\ldots,s$; it can be shown that there exist $b_{ik},\,c_{ik}\in \R$ such that
\[ |\sigma_{s+k}(a)|=\exp\left(\frac12 \sum_ib_{ik}t_i\right) \quad \text{and } \quad \sigma_{s+k}(a)=\exp\left(\frac12 \sum_ib_{ik}t_i+\sum_i c_{ik}t_i\right). \]
Now, for $(t_1,\ldots,t_s)\in\R^s$, we have
\[\phi(t_1,\ldots,t_s)=\operatorname{diag}(e^{t_1},\ldots,e^{t_s},e^{\psi_1+i\varphi_1},\ldots,e^{\psi_t+i\varphi_t}),\]
where $\psi_k=\frac12 \sum_ib_{ik}t_i$, $\varphi_k=\sum_i c_{ik}t_i$. It is clear that $G$ is a solvable Lie group.

The subgroup $U\ltimes \mathcal{O}_K$ can be embedded as a lattice in $G$; moreover, under the correspondence 
\begin{align*}
\H^s\times \C^t  & \ni (x_1+iy_1,\ldots,x_s+iy_s,z_1,\ldots,z_t) \\
& \mapsto (x_1,\log y_1,\ldots,x_s,\log y_s,z_1,\ldots,z_t) \in \R^s\times \R^s\times \C^t,
\end{align*}
the action of $U\ltimes \mathcal{O}_K$ on $\H^s\times \C^t $ can be identified with the left action of the lattice $U\ltimes \mathcal{O}_K$ on $G$. Therefore, the OT manifold 
$X(K,U)=(\H^s\times \C^t)/U\ltimes \mathcal{O}_K$ may be considered as the solvmanifold $U\ltimes \mathcal{O}_K\backslash G$.

Furthermore, it can be seen that  the natural complex structure on $G$ is left invariant and the induced complex structure on $X(K,U)$ coincides with the complex structure constructed in \cite{OT}. 

If $\g$ denotes the Lie algebra of $G$, then $\g$ has real dimension $2(s+t)$ and there is a basis $\{\alpha_1,\ldots,\alpha_s,\beta_1,\ldots,\beta_s,\gamma_1,\ldots,\gamma_{2t}\}$ of $\g$ such that the differential is given by
\begin{align*}
 d\alpha_j & =0, \\
 d\beta_j & =-\alpha_j\wedge\beta_j,\\
 d\gamma_{2k-1} & =\frac12 \sum_j b_{jk}\alpha_j\wedge\gamma_{2k-1}+\sum_j c_{jk}\alpha_j\wedge\gamma_{2k},\\
 d\gamma_{2jk} & =-\sum_j c_{jk}\alpha_j\wedge\gamma_{2k-1}+\frac12 \sum_j b_{jk}\alpha_j\wedge\gamma_{2k},
\end{align*}
for any $j=1,\ldots,s$, $k=1,\ldots,t$. Accordingly, $\g$ has a basis $\{A_1,\ldots,A_s,B_1,\ldots,B_s,C_1,\ldots,C_{2t}\}$ such that the Lie bracket is expressed as
\begin{align}\label{OT}
[A_j,B_j] &=B_j,\nonumber \\
[A_j,C_{2k-1}] &=-\frac12 b_{jk}C_{2k-1}+c_{jk}C_{2k},\\
[A_j,C_{2k}] &=-c_{jk} C_{2k-1}-\frac12 b_{jk}C_{2k},\nonumber 
\end{align}
for any $j=1,\ldots,s$, $k=1,\ldots,t$. In terms of these bases, the complex structure is given by
\[ JA_j=B_j,\qquad JC_{2k-1}=C_{2k}, \]
and a basis of $(1,0)$-forms is given by $\{\alpha_j+i\beta_j, \gamma_{2k-1}+i\gamma_{2k}\mid j=1,\ldots,s$, $k=1,\ldots,t\}$.

\smallskip

For $t=1$, consider the closed $1$-form $\theta\in\g^*$ given by $\theta=\alpha_1+\cdots+\alpha_s$, and the $2$-form $\omega\in\alt^2\g^*$ given by
\[ \omega=2\sum_i \alpha_i\wedge\beta_i + \sum_{i\neq j}\alpha_i\wedge \beta_j + \gamma_1\wedge \gamma_2. \]
Then it can be verified that $d\omega=\theta\wedge\omega$ and, therefore, for $g=\omega(\cdot,J\cdot)$, we have that $(J,g)$ is a left invariant LCK structure on $G$.

\

\begin{ejemplo}
 In \cite{S4} H. Sawai gives an explicit construction of a 6-dimensional OT solvmanifold carrying LCK metrics. He considers the $6$-dimensional group $G$ of matrices given by
\[ G=\left\{ \begin{pmatrix}
      \alpha_1^{t_1}{\alpha'}_1^{t_2} & 0 & 0 & 0 & x_1 \\
      0 & \alpha_2^{t_1}{\alpha'}_2^{t_2} & 0  & 0 & x_2 \\
      0 & 0 & \beta^{t_1}\beta'^{t_2}  & 0 & z \\
      0 & 0 & 0 & \overline{\beta}^{t_1}\overline{\beta}'^{t_2}  & \overline{z} \\
      0 & 0 & 0 & 0 & 1
     \end{pmatrix}\, : t_1,t_2,x_1,x_2\in\R,\, z\in \C\right\}, \]
where $\alpha_1,\alpha_2\in\R$ and $\beta,\overline{\beta}\in\C$ are the different roots of the polynomial $f_1(x) = x^4-2x^3-2x^2 + x + 1$, and 
$\alpha'_i=\alpha_i^{-1}+\alpha_i^{-2}$, $\beta'= \beta^{-1}+\beta^{-2}$. It can be seen that $\alpha'_1,\alpha'_2,\beta',\overline{\beta}'$ are the different roots of the 
polynomial $f_2(x) = x^4-4x^3 +4x^2-3x +1$. 

Sawai shows the existence of lattices in $G$, by carefully studying properties of the polynomials $f_1$ and $f_2$, which are the characteristic polynomials of the following 
unimodular matrices:
\[ B_1=\begin{pmatrix}
        0 & 0 & 0 & -1 \\ 1 & 0 & 0 & -1 \\ 0 & 1 & 0 & 2 \\ 0 & 0 & 1 & 2
       \end{pmatrix}, \quad 
   B_2=\begin{pmatrix}
        2 & 0 & 1 & 0 \\ 2 & 2 & 1 & 1 \\ -1 & 2 & 0 & 1 \\ 0 & -1 & 0 & 0
       \end{pmatrix}. \]
Indeed, it is easy to see that $B_1B_2=B_2B_1$ and therefore these matrices are simultaneously diagonalizable over $\C$, and let $P\in \operatorname{GL}(4,\C)$ such that 
$PB_1P^{-1}$ and $PB_2P ^{-1}$ are both diagonal. Sawai then shows that 
\[ \Gamma:=\left\{ \begin{pmatrix}
            PB_1^{s_1}B_2^{s_2}P^{-1} & P\begin{pmatrix} x_1\\x_2\\y_1\\y_2 \end{pmatrix} \\
            0 \quad 0 \quad 0 \quad 0 & 1            
           \end{pmatrix} \, : \, s_j,x_j,y_j\in \Z\right \}\]
is a lattice in $G$. 

It can be shown that the Lie brackets of the Lie algebra $\g$ of $G$ are isomorphic to the Lie brackets \eqref{OT} with $t=1$ for some constants $b_{jk},\,c_{jk}$, thus the solvmanifold $\Gamma\backslash G$ is in fact an OT solvmanifold carrying an invariant LCK structure. 
\end{ejemplo}

\

OT manifolds have been intensely studied lately; see for instance the survey \cite{OVu}. We summarize next some of their properties:
\begin{itemize}
\item The LCK structure on an OT manifold with $s=2,t=1$ is a counterexample to a conjecture formulated by I. Vaisman, which stated that any compact LCK manifold has odd first Betti number. Indeed, it was shown in \cite{OT} that this 6-dimensional LCK manifold has first Betti number $b_1=2$ (more generally, $b_1=s$ for any OT manifold of type $(s,t)$). Using the description of the OT manifolds as solvmanifolds, Kasuya proved in \cite{K1} that the Betti numbers of an OT manifold of type $(s,1)$ are given by 
$b_p=b_{2s+2-p}=\binom{s}{p}$ for $1\leq p\leq s$ and $b_{s+1}=0$. Therefore, for even $s$ these manifolds also provide counterexamples to the Vaisman's conjecture. 

\item Kasuya proved in \cite{K} that OT manifolds do not admit any Vaisman metric, by studying the Morse-Novikov cohomology $H^*_\theta$ of solvmanifolds $\Gamma\backslash G$ equipped with LCS structures, where $\theta$ denotes the corresponding Lee form of the LCS structure.

\item Ornea and Verbitsky proved in \cite{OV5} that OT manifolds of type $(s,1)$ contain no non-trivial complex submanifolds. This generalizes the fact that Inoue surfaces carry no closed analytic curve.

\item In \cite{Z} the Chern-Ricci flow, an evolution equation of Hermitian metrics, is studied on a family of OT manifolds. It is shown that, after an initial conformal change, the flow converges, in the Gromov-Hausdorff sense, to a torus with a flat Riemannian metric determined by the OT-manifolds themselves.

\item In \cite{IO} both the de Rham cohomology and the Morse-Novikov cohomology of OT manifolds of general type are studied. In fact, the cohomology groups $H_{dR}^*(X(K,U))$ and $H_{\theta}^*(X(K,U))$ (for any closed 1-form $\theta$ on $X(K,U)$) are computed in terms of invariants associated to the background number field $K$. 
\end{itemize}

\smallskip

\begin{rem}
We would like to mention that OT manifolds are the only known examples of non-Vaisman LCK solvmanifolds in dimensions greater than $4$. It would be very interesting to find new examples.
\end{rem}

\

As mentioned previously, OT manifolds with $t=1$ admit LCK metrics. The natural question is to ask about the existence of such metrics on OT manifolds with $t>1$. Oeljeklaus and Toma proved already in \cite[Proposition 2.9]{OT} that an OT manifold with $s=1$ and $t\geq 2$ admits no LCK metric. Moreover, it follows from the proof of this result that if a LCK metric exists on a OT manifolds of type $(s,t)$ then the complex embeddings satisfy the condition
\begin{equation}\label{igual_modulo}
 |\sigma_{s+1}(u)|=\cdots =|\sigma_{s+t}(u)|, \quad \forall u \in U. 
\end{equation}
This fact was reobtained by V. Vuletescu in \cite{Vu}, where he also proves that an OT manifold of type $(s,t)$ admits no LCK metric whenever $1<s\leq t$. Later, in an appendix of \cite{D} written by L. Battisti, it was shown that condition \eqref{igual_modulo} is also sufficient for the existence of a LCK metric. Therefore there is the following result:

\begin{teo}
Let $X$ be an OT manifold of type $(s,t)$. Then $X$ admits a LCK metric if and only if \eqref{igual_modulo} holds.
\end{teo}

The proof of the sufficiency of condition \eqref{igual_modulo} follows by exhibiting a K\"ahler potential on $\mathbb{H}^s \times \C^t$, given by 
\[ \varphi(z)= \left(\prod_{j=1}^s \frac{i}{z_j-\overline{z}_j} \right)^\frac{1}{t} +\sum_{k=1}^t |z_{s+k}|^2 .\] 
The corresponding K\"ahler form $\omega:=i\partial \overline{\partial}\varphi$ on $\mathbb{H}^s \times \C^t$ gives rise to a LCK metric on the associated OT manifold.

\

The main result of \cite{D} shows that if an OT manifold with $t\geq 2$ admits an LCK metric, then the natural numbers $s$ and $t$ satisfy a strong condition. Indeed, one has:

\begin{teo}
 If an OT manifold with $s\geq 1$, $t\geq 2$, admits an LCK metric then there exist integers $m\geq 0$ and $q\geq 2$ such that
 \[ s = (2t + 2m)q - 2t. \]
 In particular, $s$ is even and $s\geq 2t$.
\end{teo}

Recently, Istrati and Otiman determined in \cite{IO} the Betti numbers of OT manifolds carrying LCK metrics. Indeed, they show that the Betti numbers $b_j=\dim_\C H_{dR}^j(X,\C)$ of an OT manifold $X$ of type $(s,t)$ admitting some LCK metric are given by 
\begin{gather*}
 b_j=b_{2(s+t)-j}=\binom{s}{j}, \quad \text{if } 0\leq j\leq s,\\
 b_j=0, \quad \text{if } s<j<s+2t. 
\end{gather*}
Regarding the Morse-Novikov cohomology of OT manifolds carrying LCK metrics, it is also shown in \cite{IO} that if $X=X(K,U)$ is an OT manifold of type $(s,t)$ then there exists at most one Lee class of a LCK metric. Furthermore, if $\theta$ is the Lee form of a LCK structure on $X$, then the corresponding twisted Betti numbers $ b_j^\theta= \dim_\C H_{\theta}^j(X,\C)$ are given by:
\[ b_j^\theta= t\binom{s}{j-2}, \qquad \text{for } 0\leq j\leq 2(s+t).\]
In particular, this shows again that OT manifolds do not admit Vaisman metrics, since $H^*_\theta(X)$ does not vanish (see \cite{LLMP}).

Recalling the description of OT manifolds as solvmanifolds $\Gamma\backslash G$, one may wonder if the isomorphisms \eqref{deRham} and \eqref{kill_adapted} still hold for this class of solvmanifolds, even if $G$ is not completely solvable. For the de Rham cohomology, it is shown in \cite{K1} that \eqref{deRham} holds in the case of  OT manifolds of type $(s,1)$. For the Morse-Novikov cohomology, it was shown in \cite{AOT} that the isomorphism \eqref{kill_adapted} holds for a special class of OT manifolds of type $(s,1)$, namely, those satisfying the \textit{Mostow condition}. However, using the description of the Morse-Novikov cohomology of OT manifolds given in \cite{IO}, Istrati and Otiman show that \eqref{kill_adapted} holds for any OT manifold of type $(s,1)$, even if the Mostow condition does not hold.

\

\section{Vaisman solvmanifolds}

In this section we will consider solvmanifolds equipped with invariant Vaisman structures, which will be called simply Vaisman solvmanifolds. Recall that a LCK structure where the Lee form is parallel with respect to the Levi-Civita connection of the Hermitian metric is called Vaisman. Firstly we state some structural results proved in \cite{AO2}, and secondly we use these results in order to exhibit Vaisman solvmanifolds whose underlying complex manifold has holomorphically trivial canonical bundle. 

\subsection{Structure results}

In this section we present some results from \cite{AO2} concerning solvmanifolds equipped with invariant Vaisman structures. Since this structure is invariant, we may work at the 
Lie algebra level, and therefore we will consider a Vaisman structure $(J,\pint)$ on a Lie algebra $\g$, with associated Lee form $\theta$. We may assume that $|\theta|=1$. Therefore, if $A\in\g$ denotes the metric dual to $\theta$, we may decompose $\g$ as $\g=\R A\ltimes \ker\theta$, where $\theta(A)=1$ and $\ker\theta$ is an ideal since $d\theta=0$, which implies $\g'\subseteq \ker\theta$. The Vaisman condition $\nabla \theta=0$ is equivalent to $\nabla A=0$, and due to $d\theta=0$, this is in turn equivalent to $A$ being a Killing vector field (considered as a left invariant vector field on the associated Lie group with left invariant 
metric). Recalling that a left invariant vector field is Killing with respect to a left invariant metric if and only if the corresponding adjoint operator on the Lie algebra is skew-symmetric, we have:

\begin{lema}[\cite{AO}]\label{ad_A-antisim}
Let $(J,\pint)$ be an LCK structure on $\g$ and let $A\in\g$ be as above. Then $(J,\pint)$ is Vaisman if and only if $\ad_A$ is a skew-symmetric endomorphism of $\g$.
\end{lema}

Other properties of Vaisman Lie algebras are given in the next result:

\begin{prop}[\cite{AO2}]\label{vaisman-properties}
Let $(\g,J,\pint)$ be a Vaisman Lie algebra, then
\begin{enumerate}
\item $[A,JA]=0$,
\item $J\circ\ad_A=\ad_A\circ J$,
\item $J\circ\ad_{JA}=\ad_{JA}\circ J$,
\item $\ad_{JA}$ is skew-symmetric.
\end{enumerate}
\end{prop}

\

Let us denote $\k:=\text{span}\{A,JA\}^\perp$, so that $\ker \theta= \R JA\oplus \k$ and $\g=\R A\oplus \R JA\oplus \k$. Set $\xi:=JA$, $\eta:=-J\theta|_{\ker\theta}$ and define an 
endomorphism $\phi\in\operatorname{End}(\ker\theta)$ by $\phi(a\xi+x)=Jx$ for $a\in\R$ and $x\in \k$. It follows from general results of I. Vaisman that the quadruple 
$(\pint|_{\ker\theta},\phi,\eta,\xi)$ determines a Sasakian structure on $\ker\theta$ (up to certain constants). 

\medskip

Let us further assume that the Vaisman Lie algebra $\g$ is solvable and unimodular. In any unimodular LCK Lie algebra we have that the vector field $JA$ is in the commutator 
ideal $\g'$ of $\g$ (see \cite{AO2}). Therefore, since $\g$ is solvable, the endomorphism $\ad_{JA}$ is nilpotent. On the other hand, it follows from Proposition \ref{vaisman-properties} that 
$\ad_{JA}$ is skew-symmetric. These two conditions imply that $\ad_{JA}=0$, that is, $JA$ is a \textit{central} element of $\g$. Moreover, it was proved in \cite{AO2} that the 
dimension of the center of $\g$ is at most 2, since it is contained in $\text{span}\{A,JA\}$.

It follows that the Sasakian Lie algebra $\ker \theta$ has a one-dimensional center. Using results from \cite{AFV}, it follows that the $J$-invariant subspace $\k$ admits a Lie 
bracket $[\,,]_\k$ such that $(\k,[\,,]_\k,J|_\k,\pint|_\k)$ is a K\"ahler Lie algebra. Indeed, the Lie bracket $[\,,]_\k$ is defined as follows: for $x,y\in\k$, $[x,y]_\k$ 
denotes the component in $\k$ of $[x,y]$. Moreover, it is easy to show that we actually have
\begin{equation}\label{corchete_k}
 [x,y]= \omega(x,y) JA+[x,y]_\k, \quad x,y\in\k.
\end{equation}
This means that $\ker\theta$ is isomorphic to the one-dimensional central extension of $(\k,[\,,]_\k)$ by the cocycle $\omega|_{\k}$, which coincides with the K\"ahler form
of $(J|_\k,\pint|_\k)$. 

From now on, $\k$ will be considered as a Lie algebra with Lie bracket $[\,,]_\k$. It is easy to verify that $\k$ is also unimodular. Since $\k$ is K\"ahler, it follows from a classical result by Hano \cite{Hano} that the metric $\pint|_\k$ is \textit{flat}. In what follows a Lie algebra equipped with a flat metric will be called a flat Lie algebra.

We note moreover that if we denote $D:=\ad_A$, then $D$ satisfies: $D(\k)\subseteq \k$ and $D|_{\k}$ is a skew-symmetric derivation of $\k$ that commutes with $J|_{\k}$. This procedure can be reversed, and the main result in \cite{AO2} is:

\begin{teo}\label{Vaisman-KF}
There is a one to one correspondence between unimodular solvable Lie algebras equipped with a Vaisman structure and pairs $(\mathfrak k, D)$ where $\mathfrak k$ is a K\"ahler 
flat Lie algebra and $D$ is a skew-symmetric derivation of $\mathfrak k$ which commutes with the complex structure.
\end{teo} 

\

We continue our study of unimodular solvable Vaisman Lie algebras. Applying Milnor's result on flat Lie algebras \cite{Mi}, which was later refined in \cite{BDF}, we obtain that 
$\k$ can be decomposed orthogonally as $\k=\z\oplus\h\oplus \k'$, where $\z$ denotes the center of $\k$, $\h$ is an abelian subalgebra and $\k'$ denotes the commutator ideal of 
$\k$, which happens to be abelian and even-dimensional. Moreover, the following conditions hold, where $\nabla$ denotes the Levi-Civita connection on $\k$, which is flat:
\begin{enumerate}
	\item $\ad^\k :\h\to\mathfrak{so(k')}$ is injective,
	\item $\ad^\k_x=\nabla_x$ for any $x\in\z\oplus\h$,
	\item $\nabla_x=0$ if and only if $x\in\z\oplus\k'$.
\end{enumerate}
Furthermore, since $\k$ is K\"ahler, we have
\begin{enumerate}
 \item[(4)] $J(\z\oplus\h)=\z\oplus\h \quad \text{and} \quad J(\k')=\k'$,
 \item[(5)] $[H,Jx]_\k=J[H,x]_\k$ for any $H\in\h,\,x\in\k'$, i.e. $\ad^\k :\h\to\mathfrak{u(k')}$.
\end{enumerate}

\medskip

It was proved in \cite{AO2} that any unitary derivation of a K\"ahler flat Lie algebra acts as the zero endomorphism on $\h + J\h$, therefore we obtain that $[A,H]=[A,JH]=0$ for 
all $H\in\h$. As a consequence of these results, we were able to prove that the class of unimodular solvable Lie algebras that admit Vaisman structures is quite restricted:

\begin{teo}[\cite{AO2}]\label{Vaisman-imag.puras}
If the unimodular solvable Lie algebra $\g$ admits a Vaisman structure, then the eigenvalues of the operators $\ad_x$ with $x\in\g$ are all imaginary (some of them are 0).
\end{teo}

\smallskip

As a consequence of Theorem \ref{Vaisman-imag.puras}, together with results in \cite{S2} and Theorem \ref{nilpotent}, we reobtain the following result, proved by Sawai in \cite{S3}. Recall that a completely solvable Lie group $G$ is a solvable Lie group such that for all $x$ in its Lie algebra the operators $\ad_x$ have only real eigenvalues.

\begin{cor}
Let $G$ be a simply connected completely solvable Lie group and $\Gamma\subset G$ a lattice. If the solvmanifold $\Gamma\backslash G$ admits a Vaisman structure $(J,g)$ such that the complex structure $J$ is induced by a left invariant complex structure on $G$, then $G=H_{2n+1}\times\R$, where $H_{2n+1}$ denotes the $(2n+1)$-dimensional Heisenberg Lie group.
\end{cor}

\medskip

\begin{rem}
As stated above, the center of any unimodular solvable Lie algebra which admits a Vaisman structure is non trivial and, moreover, its dimension is at most 2. If we consider unimodular solvable Lie algebras with general LCK structures, there are examples with trivial center, but in all the examples known so far the dimension of the center is still bounded by 2. 
\end{rem}

\begin{rem}
In \cite{MP} the authors construct examples of compact Vaisman manifolds which are obtained as the total spaces of a principal $S^1$-bundle over coK\"ahler manifolds. They also show that these Vaisman manifolds are diffeomorphic to solvmanifolds. 
\end{rem}

\begin{rem}
Yet another description of unimodular Vaisman Lie algebras was given in \cite{AHK}. Indeed, in that article it is shown that applying suitable modifications to any unimodular Vaisman Lie algebra one obtains one of the following Lie algebras: $\h_{2n+1}\times\R$, $\mathfrak{su}(2)\times \R$ or $\mathfrak{sl}(2,\R)\times\R$ (see \cite{AHK} for the relevant definitions).
\end{rem}

\

\subsection{Canonical bundle of Vaisman solvmanifolds}

In this section we present original results concerning the canonical bundle of Vaisman solvmanifolds. Indeed, using the description of unimodular solvable Lie algebras given in the previous section, we exhibit in Theorem \ref{canonical} necessary and sufficient conditions for the existence of an \textit{invariant} nowhere vanishing holomorphic $(n,0)$-form on such a solvmanifold, where $n$ is half its real dimension. As a consequence, in this case the canonical bundle of the complex solvmanifold is holomorphically trivial.

We recall first that the complex manifold underlying the Vaisman nilmanifold $\Gamma\backslash H_{2n+1}\times S^1$, where $\Gamma$ is a lattice in the Heisenberg group $H_{2n+1}$, 
has holomorphically trivial canonical bundle. Indeed, it was shown in \cite{BDV} and \cite{CG} that any nilmanifold equipped with an invariant complex structure has trivial canonical bundle. A proof follows by  using a special basis of left invariant $(1,0)$-forms on the nilmanifold determined by S. Salamon in \cite{Sal}.

On the other hand, the main examples of Vaisman manifolds, that is, the Hopf manifolds $S^1\times S^{2n-1}$, have non trivial canonical bundle since its Kodaira dimension is $-\infty$. This follows from the fact that this Vaisman structure is toric (as shown in \cite[Example 4.8]{Pi}), together with the fact that any compact toric LCK manifold has Kodaira dimension $-\infty$ (\cite[Theorem 6.1]{MMP}). 

In what follows we will show that some Vaisman solvmanifolds have holomorphically trivial canonical bundle. Indeed, using the description of Vaisman Lie algebras given in the previous section, we will determine which of these Lie algebras admit a holomorphic $(n,0)$-form, where $2n$ is the real dimension of the Lie algebra. This holomorphic form will induce a nowhere vanishing section of the canonical bundle of any associated solvmanifold.

\medskip

Let $M=\Gamma\backslash G$ be a $2n$-dimensional solvmanifold equipped with an invariant Vaisman structure $(J,g)$. If $\g$ denotes the Lie algebra of $G$, which is unimodular and solvable, then $(J,g)$ determines a Vaisman structure on $\g$ with Lee form $\theta\in\g^\ast$. If $A\in\g$ denotes the metric dual to $\theta$, then we know that $\g=\R A\oplus \ker\theta$, $JA$ is a central element of $\g$, and the $J$-invariant subspace $\k=\text{span}\{A,JA\}^\perp\subset \ker\theta$ admits a Lie bracket 
that turns it into a K\"ahler Lie algebra with flat metric. Moreover, $\k$ can be decomposed orthogonally as $\k= \z\oplus \h\oplus \k'$ as before, where $\z\oplus \h$ and $\k'$ are $J$-invariant subspaces and $\ad_A$ vanishes on $\h+ J\h$.

Recall that $\ad_A$ commutes with $J$ and is skew-symmetric, so that $\ad_A\in\u(\g)$; we also have that $\ad_H|_{\k'}\in\u(\k')$.

\begin{teo}\label{canonical}
With notation as above, the $2n$-dimensional Vaisman solvmanifold $M=\Gamma\backslash G$ admits an invariant nowhere vanishing holomorphic $(n,0)$-form if and only if
\[ \ad_A\in\mathfrak{su}(\g), \quad \text{ and } \quad \ad_H|_{\k'}\in\mathfrak{su}(\k') \]
for all $H\in\h$. In particular, $M$ has holomorphically trivial canonical bundle.
\end{teo}

\begin{proof}
As $\ad_A=0$ on $\h+J\h$, we will decompose orthogonally the $J$-invariant subspace $\z\oplus \h$ as follows: $\z\oplus \h=(\z\cap J\z)\oplus (\h+J\h)$.
According to \cite{AO2}, there exists an orthonormal basis $\{e_1,\ldots,e_{2m}\}$ of $\k'$ such that $Je_{2i-1}=e_{2i}$ and for $A\in\g$ and any $H\in\h+J\h$ we have
\small
\[ \ad_A|_{\k'}=\left(\begin{array}{ccccc}       
0  & -a_1 &        &               &              \\
a_1 & 0          &        &               &              \\
&               & \ddots &               &              \\
&               &        &         0     & -a_m\\
&               &        &  a_m &     0        \\
\end{array}\right), \quad
\ad_H|_{\k'}=\left(\begin{array}{ccccc}       
	0  & -\lambda_1(H) &        &               &              \\
	\lambda_1(H) & 0          &        &               &              \\
	&               & \ddots &               &              \\
	&               &        &         0     & -\lambda_m(H)\\
	&               &        &  \lambda_m(H) &     0        \\
\end{array}\right) 
\]
\normalsize
for some $a_i\in\R$ and $\lambda_i\in (\h+J\h)^*$, $i=1,\ldots,m$.

Since $\ad_A|_{\z\cap J\z}$ is unitary, there exists an orthonormal basis $\{z_1,\ldots,z_{2r}\}$ of $\z\cap J\z$ such that $Jz_{2i-1}=z_{2i}$ and 
\[ \ad_A|_{\z\cap J\z}=\left(\begin{array}{ccccc}       
0  & -c_1 &        &               &              \\
c_1 & 0          &        &               &              \\
&               & \ddots &               &              \\
&               &        &         0     & -c_r\\
&               &        &  c_r &     0        \\
\end{array}\right), 
\]
for some $c_i\in\R$. Finally, we may choose an orthonormal basis $\{H_1,\ldots,H_{2s}\}$ of $\h+J\h$ such that $JH_{2i-1}=H_{2i}$, and we denote $\lambda_i^j=\lambda_i(H_j)$. 
To sum up, setting $B:=JA$ and taking \eqref{corchete_k} into account, we may describe the Lie bracket of $\g$ as follows:
\begin{gather*}
 [z_{2i-1},z_{2i}]=B, \quad [H_{2j-1},H_{2j}]=B, \quad [e_{2k-1},e_{2k}]=B, \\
 [A,z_{2i-1}]=c_i z_{2i}, \quad [A, z_{2i}]=-c_i z_{2i-1},\\
 [A,e_{2k-1}]=a_k e_{2k}, \quad [A, e_{2k}]=-a_k e_{2k-1},\\
 [H_{j},e_{2k-1}]=\lambda_k^{j} e_{2k}, \quad  [H_{j},e_{2k}]=-\lambda_k^{j} e_{2k-1},
% [H_{2j},e_{2k-1}]=\lambda_k^{2j} e_{2k}, \quad  [H_{2j},e_{2k}]=-\lambda_k^{2j} e_{2k-1}.
\end{gather*}
for $i=1,\ldots,r$, $j=1,\ldots,2s$ and $k=1,\ldots,m$.
If $\{\alpha, \beta, z^1,\ldots,z^{2r}, H^1,\ldots,H^{2s}, e^1,\ldots,e^{2m}\}$ denotes the dual basis of $\{A, B\}\cup\{z_i\}\cup\{ H_j\}\cup\{e_k\}$, then the Lie bracket can 
be encoded in terms of their differentials as follows:
\begin{align*}
 d\alpha & =0,\\
 d\beta & = -\sum_{i=1}^r z^{2i-1}\wedge z^{2i} -\sum_{j=1}^s H^{2j-1}\wedge H^{2j}-\sum_{k=1}^m e^{2k-1}\wedge e^{2k}  ,\\
 dz^{2i-1} & = c_i\alpha\wedge z^{2i},\\
 dz^{2i} & = -c_i\alpha\wedge z^{2i-1},\\
 dH^{2j-1} & = 0,\\
 dH^{2j} & = 0,\\
 de^{2k-1} & = \gamma_k\wedge e^{2k},\\
 de^{2k} & = -\gamma_k\wedge e^{2k-1}, 
 \end{align*}
where $\gamma_k = a_k\alpha + \sum_{l=1}^s\lambda_k^l H^l$ for any $k=1,\ldots,m$. Let us define the complex $1$-forms
\begin{align*} 
 \mu & = \alpha + i\beta  ,\\
 \xi_j & = z^{2j-1} + iz^{2j},\\
 \delta_j & = H^{2j-1} + iH^{2j},\\
 \omega_j & = e^{2j-1} + ie^{2j}.
 \end{align*}
These 1-forms are all of type $(1,0)$, so that 
\[ \eta:=\mu\wedge\xi_1\wedge\ldots\wedge\xi_r\wedge\delta_1\wedge\ldots\wedge\delta_s\wedge\omega_1\wedge\ldots\wedge\omega_m \]
is a $(n,0)$-form which generates $\alt^{(n,0)}\g^*$, where $n=1+r+s+m$. 

Therefore, we only have to check when $\eta$ is holomorphic or, equivalently, $\eta$ is closed. It can be easily verified that $d\eta=0$ if and only if 
\begin{equation}\label{sumas}
\sum_{j=1}^r c_j + \sum_{j=1}^m a_j=0 \quad \text{and} \quad \sum_{j=1}^m \lambda_j^k=0 \quad  \forall k=1,\ldots,s.
\end{equation}
It is easy to see that equations \eqref{sumas} hold if and only if $\ad_A\in \mathfrak{su}(\g)$ and $\ad_H|_{\k'}\in \mathfrak{su}(\k')$ for all $H\in\h$ (using that $\h+J\h \subseteq \h\oplus\z$), and this completes the proof.
\end{proof}

\medskip

\begin{rem}
When the $2n$-dimensional unimodular solvable Vaisman Lie algebra $\g=Lie (G)$ admits a non zero holomorphic $(n,0)$-form, then any solvmanifold $\Gamma\backslash G$ has holomorphically trivial canonical bundle, that is, this property is independent of the choice of lattice. 
\end{rem}

\

\begin{ejemplo}\label{fibrado}
It is well known that the only compact complex surfaces with holomorphically trivial canonical bundle are the $K3$ surfaces and the primary Kodaira surfaces. Moreover, the former do not admit Vaisman metrics and are not solvmanifolds, whereas the latter do admit Vaisman metrics and are in fact nilmanifolds. 

Next, using Theorem \ref{canonical}, we will exhibit examples of Vaisman solvmanifolds with holomorphically trivial canonical bundle in any dimension greater than or equal to $6$, which are not nilmanifolds. Indeed, in \cite{AO2} we determined all Vaisman Lie algebras $\g$ such that the associated K\"ahler flat Lie algebra $\k$ is abelian, $\k=\R^{2n-2}$. The Vaisman Lie algebra $\g$ can be 
decomposed as $\g=\R A \ltimes_D \h_{2n-1}$ where the adjoint action of $\R$ on $\h_{2n-1}$ is given by
\begin{equation}\label{D} 
D=\left(\begin{array}{cccccc}       
0 &      &      &        &      &     \\
  &   0  & -a_1 &        &      &     \\
  &  a_1 & 0    &        &      &     \\
  &      &      & \ddots &      &     \\
  &      &      &        &   0  & -a_{n-1}\\
  &      &      &        &  a_{n-1} &  0  \\
\end{array}\right), \end{equation}
in a basis $\{B=JA,e_1,\dots,e_{2n-2}\}$ of $\h_{2n-1}$ such that $[e_{2i-1},e_{2i}]=B$ and $Je_{2i-1}=e_{2i}$, for some $a_i\in\R$ (and not all of them zero, otherwise we obtain 
a nilpotent Lie algebra). We proved that whenever $a_i\in\Q$ the corresponding simply connected Lie group admits non-nilpotent lattices. It follows from Theorem \ref{canonical} that if the constants $\{a_i\}$ satisfy $\sum_{i=1}^{n-1} a_i=0$, then the associated solvmanifolds have holomorphically trivial canonical bundle and they are not diffeomorphic to a nilmanifold. For instance, when $n=3$, $a_1=1$ and $a_2=-1$, we reobtain the Lie algebra $\g_6=A_{6,82}^{0,1,1}$ from \cite{FOU}, where the classification of $6$-dimensional solvable Lie algebras equipped with a non-zero holomorphic $(3,0)$-form was given.
\end{ejemplo}

\begin{rem}
If in the construction given in Example \ref{fibrado} we take $n=1$ and $a_1=1$, the associated Lie group admits a lattice such that the corresponding $4$-dimensional solvmanifold is a secondary Kodaira surface. Therefore this secondary Kodaira surface is a Vaisman solvmanifold and it is well known that its canonical bundle is not holomorphically trivial (see for instance \cite{BHPV}). 
\end{rem}

\begin{rem}
We recall that, according to \cite{To}, a $2n$-dimensional compact complex manifold has holomorphically trivial canonical bundle if and only if it admits a Hermitian metric whose Chern connection has (unrestricted) holonomy group contained in $SU(n)$.
\end{rem}

\

\section{ Locally conformally symplectic solvmanifolds}

A natural generalization of LCK manifolds are the locally conformally symplectic (LCS) manifolds, that is, a manifold $M$ with a non degenerate $2$-form $\omega$ such that there exists an open cover $\{U_i\}$ and smooth functions $f_i$ on $U_i$ such that 
\[\omega_i=\exp(-f_i)\omega\] 
is a symplectic form on $U_i$. This condition is equivalent to requiring that
\begin{equation}\label{lcs}
d\omega=\theta\wedge\omega
\end{equation} 
for some closed $1$-form $\theta$, called the Lee form. 
Moreover, $M$ is called globally conformally symplectic (GCS) if there exist  a $C^{\infty}$ function, $f:M\to\R$, such that $\exp(-f)\omega$ is a symplectic form. Equivalently, $M$ is a GCS 
manifold if there exists a exact $1$-form $\theta$ globally defined on $M$ such that 
$d\omega=\theta\wedge\omega$.   
The pair $(\omega, \theta)$ will be called a LCS structure on $M$. Note that if $\theta=0$ then $(M,\omega)$ is a symplectic manifold. The unique vector field $V$ on $M$ satisfying $i_V\omega=\theta$ is called the Lee vector field of the LCS structure.

These manifolds were considered by Lee in \cite{L} and they have been thoroughly studied by Vaisman in \cite{V}. Some recent results can be found in \cite{Ba,Baz1,BM1,Ha,HR,LV}, among others. LCS manifolds play an important role in mathematics and in physics as well, for example in Hamiltonian mechanics, since they provide a generalization of the usual description of the phase space in terms of symplectic geometry (see \cite{V}). 

We recall a definition due to Vaisman (\cite{V}) concerning different kinds of LCS structures. Given a LCS structure $(\omega, \theta)$ on $M$, a vector field $X$ is called an infinitesimal automorphism of $(\omega,\theta)$ if 
$\textrm{L}_X\omega=0$, where $\textrm{L}$ denotes the Lie derivative. It follows that $\textrm{L}_X\theta=0$, as a consequence, $\theta(X)$ is a constant function on $M$. The subspace $\mathfrak{X}_\omega(M)=\{X\in\mathfrak{X}(M): \textrm{L}_X\omega=0\}$ is a Lie subalgebra of $\mathfrak{X}(M)$, and the map $\theta : \mathfrak{X}_\omega(M) \to \R$ is a well defined Lie algebra morphism called the Lee morphism. 
If there exists an infinitesimal automorphism $X$ such that $\theta(X)\neq 0$, the LCS structure $(\omega,\theta)$ is said to be of {\em the first kind}, and it is of {\em the second kind} otherwise. Much more is known about manifolds carrying LCS structures of the first kind. For instance, in \cite{V} it was shown the existence of distinguished foliations on such manifolds, while in \cite{Ba} it was shown that a compact manifold equipped with a LCS structure of the first kind fibers over the circle and that the fiber inherits a contact structure. It is well known that the LCS structure underlying any Vaisman manifold is of the first kind (see for instance \cite{CDMY1}).

\medskip 

The relation between LCK and LCS manifolds is analogous to the K\"ahler vs symplectic manifolds.  In \cite{OV} the following problem was formulated: 

\smallskip

\noindent ``Give examples of compact locally conformally symplectic manifolds which admit no locally conformally K\"ahler metric.''

\smallskip

The first example was given in \cite{BK}; this example is a product $M\times S^1$, where $M$ is certain $3$-dimensional compact contact manifold. More recently, in \cite{BM1} another example was constructed, this example is a $4$-dimensional nilmanifold and it is not a product of a $3$-dimensional compact manifold with $S^1$. This motivates the study of nilmanifolds or, more generally, solvmanifolds, equipped with invariant LCS structures.

\medskip

Given a Lie group $G$, a LCS structure on $G$ is called left invariant if $\omega$ is left invariant, and this easily implies that $\theta$ is also left invariant using condition \eqref{lcs}. This fact allows us to define a LCS structure on a Lie algebra. We say that a Lie algebra $\g$ admits a \textit{locally conformally symplectic} (LCS) structure if there exist $\omega\in\alt^2\g^*$ and $\theta \in \g^*$, with $\omega$ non degenerate and $\theta$ closed, such that \eqref{lcs} is satisfied.

LCS structures on Lie algebras can also be distinguished in two different kinds, just as LCS structures on manifolds. Indeed, given a LCS structure $(\omega,\theta)$ on a Lie algebra $\g$, let us denote by $\g_\omega$ the set of infinitesimal automorphisms of the LCS structure, that is, 
\begin{equation}\label{autom}
\g_\omega = \{x\in\g: \textrm{L}_x\omega=0\} = \{x\in\g: \omega([x,y],z)+\omega(y,[x,z])=0 \; \text{for all} \; y,z\in\g\}.
\end{equation}
Note that $\g_\omega \subset \g$ is a Lie subalgebra, and the restriction of $\theta$ to $\g_\omega$ is a Lie algebra morphism called the Lee morphism. The LCS structure $ (\omega,\theta)$ is said to be \textit{of the first kind} if the Lee morphism is surjective, and \textit{of the second kind} if it is identically zero. Note that a LCS structure of the first kind on a Lie algebra naturally induces a LCS structure of the first kind on any quotient of the associated simply connected Lie group by a discrete subgroup.

\smallskip

There is another way to distinguish LCS structures. Indeed, a LCS structure $(\omega,\theta)$ on a Lie algebra $\g$ is said to be \textit{exact} if $[\omega]_\theta =0$, and \textit{non exact} if $[\omega]_\theta\neq0$. It was shown in \cite{BM} that on any unimodular Lie algebra a LCS structure is of the first kind if and only if it is exact.

\

Recently, in \cite{BM} and \cite{ABP} many interesting results on LCS structures of the first kind on Lie algebras and solvmanifolds were established. We review some of them next.

\medskip

Firstly, in \cite{BM} it is shown a relation between Lie algebras equipped with LCS structures of the first kind and Lie algebras with contact structures. Recall that a $1$-form $\eta$ on a $(2n+1)$-dimensional Lie algebra $\h$ is called a contact form if $\eta\wedge (d\eta)^n\neq0$. This form $\eta$ is called the contact form and the unique vector $V\in\h$ satisfying $\eta(V)=1$ and $i_V d\eta=0$ is called the Reeb vector of the contact structure. 

\begin{teo}[\cite{BM}]\label{lcs_contact}
There is a one to one correspondence between $(2n+1)$-dimensional contact Lie algebras $(\h,\eta)$ endowed with a derivation $D$ such that $\eta\circ D=0$ and $(2n+2)$-dimensional LCS Lie algebras  of the first kind $(\g,\omega,\theta)$. The relation is given by $\h=\ker\theta$, $\omega=d_\theta \eta$ and $D=\ad_U$ where $U$ satisfies $\eta=-i_U\omega$.
\end{teo}

When the Lee vector is in the center of the Lie algebra, a more refined result can be obtained, relating the LCS structure with a symplectic structure on a lower dimensional Lie algebra (cf. Theorem \ref{Vaisman-KF}):

\begin{teo}[\cite{BM}]\label{lcs_symplectic}
There is a one to one correspondence between Lie algebras of dimension $2n +2$ admitting a LCS structure of the first kind with central Lee vector and symplectic Lie algebras $(\mathfrak{s}, \beta)$ of dimension $2n$ endowed with a derivation $E$ such that $\beta(EX,Y)+\beta(X,EY)=0$ for all $X,Y\in\mathfrak s$.
\end{teo}

\medskip

Considering LCS structures on nilpotent Lie algebras, it was shown in \cite{BM} that any such structure is of the first kind and, moreover, the Lee vector is central. Combining this with the previous theorem the next result follows:

\begin{teo}[\cite{BM}]
Every LCS nilpotent Lie algebra of dimension $(2n+2)$ with non-zero Lee form is a suitable extension of a $2n$-dimensional symplectic nilpotent Lie algebra $\mathfrak{s}$ equipped with a symplectic nilpotent derivation. 
\end{teo}

Moreover, the symplectic nilpotent Lie algebra $\mathfrak s$ may obtained by a sequence of $(n-1)$ symplectic double extensions by nilpotent derivations from the abelian Lie algebra of dimension $2$, according to \cite{MR,DM}.

\

On the other hand, LCS structures of the second kind are less understood. Examples of LCS structures of the second kind on the $4$-dimensional
solvmanifold constructed in \cite{ACFM} are given in \cite{Ba2}. More precisely, these LCS structures are non exact. In \cite{ABP} properties of non exact LCS structures were investigated, producing new examples. LCS structures on almost abelian Lie algebras were studied in \cite{AO1} and, among other results, it was proved that on any such Lie algebra of dimension at least $6$ any LCS structure is of the second kind. Moreover, for each $n\geq 2$ a $2n$-dimensional simply connected almost abelian Lie group with left invariant LCS structures of the second kind are constructed; these groups admit lattices and we obtain in this way solvmanifolds with LCS structures, moreover, it can be seen that these structures are of the second kind.

\smallskip

In \cite{ABP} a classification of LCS structures on $4$-dimensional Lie algebras, up to Lie algebra automorphism, is given and LCS structures on compact quotients of all $4$-dimensional simply connected solvable Lie groups are constructed. They also give structure results for LCS Lie algebras inspired by the results of \cite{Ov} and \cite{BM}. The first result is about exact lcs structures, not necessarily of the first kind.

\begin{teo}\cite[Theorem 1.4]{ABP}
There is a one-to-one correspondence between exact LCS Lie algebras $(\g,\omega=d\eta-\theta\wedge\eta,\theta)$, $\dim \g=2n$, such that $\theta(U)\neq0$, where $\eta=-i_U\omega$ and contact Lie algebras $(\h,\eta)$, $\dim \h=2n-1$, with a derivation $D$ such that $D^*\eta=\alpha\eta$, $\alpha\neq1$.
\end{teo}

\medskip

The second structure result is about Lie algebras with a non exact LCS structure, establishing a relation with \textit{cosymplectic} Lie algebras. We recall that a $(2n-1)$-dimensional Lie algebra $\h$ is said to be cosymplectic if there exist two closed $1$-forms $\eta,\beta\in\h^*$ such that $\eta\wedge \beta^{n-1}\neq 0$. 

\begin{teo}\cite[Proposition 1.8]{ABP}
Let $(\h,\eta,\beta)$ be a cosymplectic Lie algebra of dimension $2n-1$, endowed with a derivation $D$ such that $D^*\beta=\alpha\beta$ for some $\alpha\neq0$. Then $\g=\h\ltimes_D \R$ admits a natural LCS structure. The Lie algebra $\g$ is unimodular if and only if $\h$ is unimodular and $D^*\eta=-\alpha(n-1)\eta+\delta$ for some $\delta\in \h^*$ with $\delta(R)=0$ where $R$ denotes the Reeb vector of the cosymplectic structure. If $\h$ is unimodular then the LCS structure on $\g$ is not exact.
\end{teo}
 
\medskip

The third result is related with the cotangent extension problem in the LCS setting; the symplectic case was studied in \cite{Ov}.

\begin{teo}\cite[Proposition 1.17]{ABP}
Let $(\g,\omega,\theta)$ be a $2n$-dimensional LCS Lie algebra with a Lagrangian ideal $j\subset \ker\theta$. Then $\g$ is a solution of the cotangent extension problem, that is, there exists a $n$-dimensional Lie algebra $\h$ with a closed 1-form $\mu\in\h^*$ such that $\g\cong \h^*\oplus\h$, and $(\omega,\theta)$ is equivalent to a canonical LCS structure on $\h^*\oplus\h$ constructed using $\mu$.
\end{teo}

Furthermore, in \cite{ABP} it is shown that every LCS structure on a $4$-dimensional Lie algebra  can be recovered by one of the three constructions detailed above.

\

\end{document}